\documentclass[12pt]{amsproc}
\usepackage{amsmath,amsfonts,amssymb}

\usepackage{faktor}
\usepackage{color}
\usepackage{relsize}

\newtheorem{theorem}{\sc Theorem}
\newtheorem{lemma}[theorem]{\sc Lemma}
\newtheorem{proposition}[theorem]{\sc Proposition}
\newtheorem{corollary}[theorem]{\sc Corollary}

\usepackage{tikz}
\usepackage{tikz-cd}

\begin{document}
\title{Automorphisms of $\kappa$-existentially closed groups}
\author{Burak KAYA}
\address{ Department of Mathematics, \\ Middle East Technical University,
 06800, Ankara, Turkey }
\email{burakk@metu.edu.tr}
\author{ Mahmut KUZUCUO\u{G}LU \\
 \today}
\address{ Department of Mathematics \\ Middle East Technical University,
 06800, Ankara, Turkey.   } \email{matmah@metu.edu.tr}


\thanks{This research project was partially supported by  Middle East Technical University Research Grant GAP-101-2018-2737, Ankara, Turkey.}

\dedicatory{Dedicated to Pavel Shumyatsky for his 60$^{\text{th}}$ birthday}

\maketitle

\begin{abstract} We investigate the automorphisms of some $\kappa$- existentially closed groups. In particular, we prove that $Aut(G)$ is the union of subgroups of level preserving automorphisms and $|Aut(G)|=2^{\kappa}$ whenever $\kappa$ is inaccessible and $G$ is the unique $\kappa$-existentially closed group of cardinality $\kappa$. Indeed, the latter result is a byproduct of an argument showing that, for any uncountable $\kappa$ and any group $G$ that is the limit of regular representation of length $\kappa$ with countable base, we have $|Aut(G)|=\beth_{\kappa+1}$, where $\beth$ is the beth function. Such groups are also $\kappa$-existentially closed if $\kappa$ is regular. Both results are obtained by an analysis and classification of level preserving automorphisms of such groups.
\end{abstract}

\section{Introduction}

Let $\kappa$ be an infinite cardinal. A group $G$ with $|G| \geq \kappa$ is said to be $\kappa$-\textit{existentially closed} if every system of less than $\kappa$-many equations and inequations with coefficients in $G$ which has a solution in some supergroup $H \geqslant G$ already has a solution in $G$. We denote the class of $\kappa$-existentially closed
 groups by $\mathcal{E}_{\kappa}$.
Obviously $\mathcal{E}_{\lambda}\subseteq \mathcal{E}_{\kappa}$ for all cardinals $\kappa \leq \lambda$ and $\mathcal{E}_{\mu}=\bigcap_{\kappa<\mu} \mathcal{E}_{\kappa}$ for all limit cardinals $\mu$.

The study of $\aleph_0$-existentially closed groups was initiated by W. R. Scott in \cite{Scott51}. Scott defined the notion of an $\aleph_0$-existentially closed group and its extensions to arbitrary uncountable cardinals. $\aleph_0$-existentially closed groups are well-studied, for example, see \cite{HigmanScott88}. On the other hand, not much attention has been paid to $\kappa$-existentially closed groups for uncountable $\kappa$ until recently. The principal properties of $\kappa$-existentially closed groups were studied in \cite{KK18} and the conditions for their existence were established in \cite{KKK18}.

While it may seem at first sight that the study of $\kappa$-existentially closed groups for uncountable $\kappa$ is a mere generalization, there are some essential differences between the countable and uncountable cases. For example, that there are uncountably many $\aleph_0$-existentially closed groups of cardinality $\aleph_0$ was proven by B. H. Neumann in \cite{Neumann73}. This follows from his result that there are uncountably many $2$-generated groups \cite[Theorem 14]{Neumann37} and the result of Scott \cite{Scott51} that every group is contained in an $\aleph_0$-existentially closed group. Unlike the countably infinite case, it is shown in \cite[Theorem 2.7]{KK18} that if there exists a $\kappa$-existentially closed group of cardinality $\kappa$ for an uncountable $\kappa$, then it is unique. We would like to note that uniqueness results can also be obtained for special subclasses for $\aleph_0$-existentially closed groups. In \cite[Theorem 9]{KKK18}, we stated that any two countable $\aleph_0$-existentially closed groups that have local systems consisting of simple finitely presented subgroups are isomorphic. One can generalize this argument to the groups that have local systems consisting of finitely absolutely presented subgroups which is introduced by Neumann in \cite{Neumann73}.

B. H. Neumann in  \cite[Page 555]{Neumann73} stated that ``however, no $\aleph_0$-existentially closed group is explicitly known, the existence proof being highly non-constructive. This stem in part from the fact that there is no useful criterion known that tells one what sentences are or are not consistent over a given group." An explicit $\kappa$-existentially closed group for uncountable regular $\kappa$ was constructed in \cite{KK18} as a limit of regular representations. This group is clearly $\aleph_0$-existentially closed as $\mathcal{E}_{\kappa} \subseteq \mathcal{E}_{\aleph_0}$.

We say that a group is \textit{the limit of regular representations of length $\kappa$ with base $G_0$} if it is the direct limit of the directed system $\{G_{\alpha}\}_{\alpha<\kappa}$ where \begin{itemize}
\item $G_{\alpha+1}=Sym(G_{\alpha})$ and $G_{\alpha} \hookrightarrow G_{\alpha+1}$ via its right regular representation for all $\alpha<\kappa$, and
\item $G_{\gamma}=\displaystyle \varinjlim_{\substack{\alpha<\gamma}} G_{\alpha}$ for all limit $\gamma < \kappa$.
\end{itemize}
To avoid trivialities, we assume that $|G_0| \geq 3$. For a study of such groups, we refer the reader to \cite{Kegel09}. We have the following fact.

\begin{theorem}\cite[Section 4]{KK18} If $\kappa$ is uncountable and regular, then any limit of regular representations of length $\kappa$ is $\kappa$-existentially closed.
\end{theorem}

The aim of this paper is to investigate automorphisms of such $\kappa$-existentially closed groups. Before we proceed, let us make some notational remarks and explain some left-right multiplication conventions that are used in this paper. In what follows, the letters $\alpha,\beta,\delta,\dots$ will always denote ordinals and $\kappa,\lambda,\mu,\dots$ will denote infinite cardinals. For a group $G$, we define the conjugation map by $g \in G$ to be $\iota_g(x)=x^g=gxg^{-1}$ so that $\iota_{gh}(x)=(\iota_g \circ \iota_h)(x)$. We consider $Aut(G)$ as a group with the multiplication operation given by $(\varphi \boldsymbol{\cdot} \psi)(x)=(\varphi \circ \psi)(x)=\varphi(\psi(x))$.

We will first recall some basic facts on the ``local" behavior of automorphisms of $\kappa$-existentially closed groups that were implicitly contained in \cite{KK18}.

Let $H$ be an arbitrary group. An automorphism $\varphi \in Aut(H)$ is called $\kappa$-\textit{inner} if for every $X \subseteq H$ with $|X|<\kappa$,  there exists an element $g \in H$ such that $\varphi(x)=\iota_g(x)$ for all $x \in X$. Let $\kappa$-$Inn(H)$ denote the set of all $\kappa$-inner automorphisms of $H$. We clearly have $Inn(H) \unlhd \kappa$-$Inn(H) \unlhd Aut(H)$. Moreover, the inclusion on right is indeed an equality for $\kappa$-existentially closed groups.

\begin{proposition}\label{k-inner2} Let $\kappa$ be uncountable and let $G$ be $\kappa$-existentially closed. Then every automorphism of $G$ is $\kappa$-inner, i.e. $\kappa$-$Inn(G)=Aut(G)$.
\end{proposition}
\begin{proof} This easily follows from \cite[Lemma 2.4]{KK18}.
\end{proof}

\begin{proposition}\label{k-inner5} Let $\kappa$ be uncountable and let $G$ be a $\kappa$-existentially closed group. Then every automorphism of $H \leqslant G$ with $|H|<\kappa$ can be extended to an automorphism of $G$.
\end{proposition}
\begin{proof} This is an immediate consequence of Proposition \ref{k-inner2}.
\end{proof}

We wish to note that Proposition \ref{k-inner2} fails for $\kappa=\aleph_0$. Nevertheless, if one requires the set $X$ in the definition of $\kappa$-innerness to be a subgroup and not just a subset, then this facts also holds for $\aleph_0$-existentially closed groups. More specifically, we have the following fact which is a consequence of \cite[Lemma 2.4]{KK18}.

\begin{proposition} Let $G$ be an $\aleph_0$-existentially closed group.
Then for every $\varphi \in Aut(G)$ and for every finite subgroup $A \leqslant G$, there exists an element $g \in G$ such that $\varphi(a)=a^g$ for all $a\in A$.
\end{proposition}

It turns out that embeddings of small groups into a $\kappa$-existentially closed group can be extended to embeddings of their small supergroups.

\begin{lemma}\label{k-inner4}
Let $G$ be a $\kappa$-existentially closed group.
Let $H\leq K$ be groups with $|K|<\kappa$. Then for every embedding $\varphi: H \rightarrow G$, there exists an embedding $\overline{\varphi}: K \rightarrow G$  such that $\overline{\varphi} \upharpoonright H=\varphi$.
\end{lemma}

\begin{proof}  Let $\varphi: H \rightarrow G$ be an embedding. Pick an embedding $\theta: K \rightarrow G$, which exists by \cite[Lemma 2.1]{KK18}. Then $\varphi(H)$ and $\theta(H)$ are isomorphic subgroups of $G$ of order less than $\kappa$. Hence, by \cite[Lemma 2.4]{KK18}, there exists an element $t\in G$ satisfying $$(\theta(h))^{t}=\varphi(h)$$ for all $h \in H$. Then the map $\overline{\varphi}: K \rightarrow G$ given by $\overline{\varphi}(x)=(\iota_t \circ \theta)(x)$ is an extension of $\varphi: H \rightarrow G$.
\end{proof}

Using this lemma, one can obtain $\kappa$-existentially closed groups as limits of directed systems consisting of groups that contain arbitrarily large infinite symmetric groups. To demonstrate this, suppose that $\kappa$ is inaccessible and let $G$ be a limit of regular representations of length $\kappa$ with countable base $G_0$. In this case, $G$ is the unique $\kappa$-existentially closed group of cardinality $\kappa$. We shall now obtain this group as a limit of general linear groups.

Let $F$ be a fixed field with $|F|<\kappa$. Recall that the right regular representation of a group $H$ with $|H|=\mu$ can be seen as a subgroup $GL(\mu, F)$. Consequently, for each $\alpha < \kappa$, by Lemma \ref{k-inner4}, we can find an embedding $\overline{\varphi_{\alpha}}$ such that
\[\begin{tikzcd}
G_{\alpha} \arrow[r, hook] \arrow[dr, hook]
& GL(|G_{\alpha}|,F) \arrow[d, "\overline{\varphi_{\alpha}}"]\\
& G
\end{tikzcd}
\]
commutes. Since $|GL(|G_{\alpha}|,F)|=|Sym(G_{\alpha})|=|G_{\alpha+1}|<\kappa$, we can find some $G_{\alpha'}$ containing the image of $\overline{\varphi_{\alpha}}$. Then the same procedure can be applied to $G_{\alpha'}$. Repeating this procedure transfinitely along $\alpha<\kappa$ by taking direct limit at limit stages, one can obtain the following.

\begin{corollary}\label{projective-linear} Let $G$ be a $\kappa$-existentially closed group of cardinality $\kappa$ where $\kappa$ is inaccessible. Then $G$ is the direct limit of some directed system consisting of $GL(\mu_{\alpha},F)$'s for non-limit $\alpha<\kappa$ where $\mu_{\alpha}<\kappa$ are cardinals with $\kappa=\sup\{\mu_\alpha:\alpha<\kappa\}$.
\end{corollary}

One can replace $GL(\mu,F)$ by $PGL(\mu,F)$ since it also embeds the right regular representation of a group of cardinality $\mu$.

Before we conclude this section, we wish record the following fact for $\kappa$-existentially closed groups.

\begin{proposition}\label{relations} Let $G$ be a $\kappa$-existentially closed group whose set of relations has cardinality $\mu$. Then $\mu \geq \kappa$.
\end{proposition}
\begin{proof} Assume to the contrary that $\mu < \kappa$. Set $S$ to be the set of elements of $G$ which do not appear in any of the relations of $G$ and let $T=G-S$. Then $\langle S \rangle$ is a free subgroup of $G$ of rank $|S| \geq \kappa$ and $\langle T \rangle$ is of cardinality $\leq \mu$. Moreover $G$ is a free product of  $\langle T \rangle$ and $\langle S \rangle$. But, for any subset $A\subseteq  S$ with $2\leq | A |<\kappa $, the centralizer of $A$ in the free group $\langle S \rangle$ is trivial and hence the centralizer in the free product is trivial. But this is impossible, because in a $\kappa$-existentially closed group $G$ any subgroup of order less than $\kappa$ has a non-trivial centralizer in $G$, see \cite[Lemma 3.5]{KK18}.
\end{proof}

\section{Level preserving automorphisms}

In this section, we will focus on the ``global" behavior of automorphisms of $\kappa$-existentially closed groups that are limit of regular representations. To set the scene, let us start with weaker assumptions.

Through the rest of the paper, suppose that $G$ is a group of cardinality $\kappa$ such that $G=\bigcup_{\alpha < \kappa} G_{\alpha}$ for some sequence of groups $(G_{\alpha})_{\alpha < \kappa}$ with 
\begin{itemize}
\item $G_{\alpha} \leqslant G_{\beta}$ for all $\alpha < \beta < \kappa$,
\item $G_{\gamma}=\cup_{\alpha < \gamma} G_{\alpha}$ for limit $\gamma<\kappa$,\ \ \ \ \ \ \ \ \ \ \ \ \ \ $\mathlarger{\mathlarger{\mathlarger{(\ast)}}}$
\item $|G_{\alpha}|<\kappa$ for all $\alpha < \kappa$.
\end{itemize}
We wish to understand the automorphisms of $G$. To this end, we begin by noting the following simple but crucial fact which roughly states that every automorphism of $G$ set-wise fixes a ``big" set of levels.

\begin{lemma}\label{fixedisclub} Suppose that $\kappa$ is uncountable and regular. For every $\varphi \in Aut(G)$, we have that $$Stab(\varphi)=\{\alpha<\kappa: \varphi[G_{\alpha}]=G_{\alpha}\}$$ is a club (i.e. closed and unbounded) subset of $\kappa$.
\end{lemma}
\begin{proof} Let $\varphi \in Aut(G)$. We first show that $Stab(\varphi)$ is unbounded in $\kappa$. Let $\alpha < \kappa$. We recursively construct an increasing sequence $(\alpha_n)_{n < \omega}$ of ordinals below $\kappa$ as follows. Set $\alpha_0=\alpha$. Suppose that $\alpha_n$ has already been constructed. Let $\alpha_{n+1}$ be any $\gamma$ greater than $\alpha_n$ such that $$\varphi[G_{\alpha_n}] \cup \varphi^{-1}[G_{\alpha_n}] \subseteq G_{\gamma}$$ Observe that, if there were no such $\gamma<\kappa$, then we would have \[cf(\kappa) \leq \big|\varphi[G_{\alpha_n}] \cup \varphi^{-1}[G_{\alpha_n}]\big|<\kappa\] which contradicts the regularity of $\kappa$. Thus this recursive construction is possible. Now choose $\beta=\sup\{\alpha_n:n <\omega\}$. Then $\alpha<\beta<\kappa$ since $cf(\beta)=\omega<\kappa=cf(\kappa)$. Let $x \in G_{\beta}$. Then we have $x \in G_{\alpha_k}$ for some $k<\omega$. By construction, we have $\varphi(x),\varphi^{-1}(x) \in G_{\alpha_{k+1}} \subseteq G_{\beta}$. It follows that $\varphi[G_{\beta}]=G_{\beta}$ and hence, $Stab(\varphi)$ is unbounded in $\kappa$.

Next will be shown that $Stab(\varphi)$ is closed in $\kappa$. Let $\gamma < \kappa$ and $(x_{\xi})_{\xi<\gamma}$ be an increasing sequence of elements of $Stab(\varphi)$. Consider $\delta=\sup\{x_{\xi}:\xi<\gamma\}$. Then $\delta$ is a limit ordinal and hence
\[ G_\delta=\bigcup_{\xi<\gamma} G_{x_{\xi}}=\bigcup_{\xi<\gamma} \varphi[G_{x_{\xi}}]=\varphi\left[\bigcup_{\xi<\gamma} G_{x_{\xi}}\right]=\varphi[G_{\delta}]\]
implying that $\delta \in Stab(\varphi)$. Thus $Stab(\varphi)$ is closed.\end{proof}

Motivated by Lemma \ref{fixedisclub}, we now introduce the notion of a level preserving automorphism. Let $C \subseteq \kappa$. An automorphism $\varphi \in Aut(G)$ is said to be $C$-\textit{level preserving} if \[\varphi[G_{\alpha}] = G_{\alpha}\]
for all $\alpha \in C$. We shall denote the set of $C$-level preserving automorphisms of $G$ by $Aut_C(G)$. We clearly have $Aut_{\emptyset}(G)=Aut(G)$ and $Aut_C(G) \leqslant Aut_D(G)$ whenever $D \subseteq C$.

Our main goal in this section is to understand the structure of $Aut_C(G)$ under additional hypotheses on $C$ and $G$. Before we proceed to do that, we first derive an important consequences of Lemma \ref{fixedisclub} which says that each ``small" subgroup of $Aut(G)$ is contained some $Aut_C(G)$. More precisely, we have the following.

\begin{corollary} Suppose that $\kappa$ is uncountable and regular. For every $H \leqslant Aut(G)$ with $|H|<\kappa$, there exists a club set $C \subseteq \kappa$ with $H \leqslant Aut_C(G)$
\end{corollary}
\begin{proof} Let $H \leqslant Aut(G)$ be with $|H|<\kappa$. By Lemma \ref{fixedisclub}, $Stab(\varphi) \subseteq \kappa$ is club for every $\varphi \in Aut(G)$. It then follows from \cite[Theorem 8.3]{Jech03} that the set
\[ C=\bigcap_{\varphi \in H} Stab(\varphi)\]
is a club subset of $\kappa$. By construction, we have $H \leqslant Aut_{C}(G)$.
\end{proof}

\begin{corollary} Suppose that $\kappa$ is uncountable and regular. Then $Aut(G)$ is the union of level preserving automorphisms.
\end{corollary}
\begin{proof}$\displaystyle Aut(G)=\bigcup_{\substack{H \leqslant Aut(G) \\ |H|<\kappa}} H = \bigcup_{\substack{C \subseteq \kappa \\ C \text{ is club}}} Aut_C(G)$.\end{proof}

In particular, this corollary applies to $\kappa$-existentially closed groups.

\begin{corollary} Let $\kappa$ be inaccessible and let $K$ be the unique $\kappa$-existentially closed group of cardinality $\kappa$, which (necessarily) is a limit of regular representations of length $\kappa$ with countable base. Then
$$ Aut(K)= \bigcup_{\substack{C \subseteq \kappa \\ C \text{ is club}}} Aut_C(K)=\bigcup_{\alpha<\kappa} Aut_{\{\alpha\}}(K)$$
\end{corollary}

We would like to note that Lemma \ref{fixedisclub} and its first two corollaries do not have anything to do with the group structure of $G$, but rather, are related to the global combinatorial structure of $G$. Indeed, these results hold for arbitrary structures in the model-theoretic sense if one replaces ``subgroups" with ``substructures".

The hypothesis that $\kappa$ is uncountable and regular is essential to these results. For example, Hall's universal locally finite group is a countable increasing union of finite groups but it has automorphisms that are not level preserving; see \cite{Rabin77}.

We shall next focus on determining the group $Aut_C(G)$ which requires us to assume more about the group structure of $G$. For our purposes, suppose for the remainder of the paper that $G$ is a limit of regular representations of length $\kappa$ with base $Sym(n)$ for some $n \geq 7$, i.e. it satisfies the additional properties that
\begin{itemize}
\item $G_0 = Sym(n)$ for some $n \geq 7$.
\item $G_{\alpha} \hookrightarrow Sym(G_{\alpha}) = G_{\alpha+1}$ is embedded via its right regular representation.
\end{itemize}

We would like to note that these assumptions on $G$ together with properties $(\ast)$ imply that $\kappa$ is an inaccessible cardinal whenever it is uncountable and regular.

The group $G_{\omega}$ is known as Hall's universal locally finite group and is well-studied. For example, we refer the reader to \cite[Chapter 6]{KW73} for some of its properties.

Fix a club subset $A \subseteq \kappa$ and let $\mathbf{a}=(a_{\alpha}:\alpha<\kappa)$ be the (unique) increasing sequence of length $\kappa$ which enumerates $A$. Observe that, if $\gamma<\kappa$ is limit, then we have $\sup\{a_{\alpha}:\alpha<\gamma\} \in A$ because $A$ is closed, and hence, $a_{\gamma}=\sup\{a_{\alpha}:\alpha<\gamma\}$. So $\mathbf{a}$ has limit ordinals at indices that are limit ordinals.

Consider the set
\[ \mathcal{G}_A=\prod_{\substack{\alpha  <\kappa\\\alpha \text{ is non-limit}}} \mathbf{C}_{G_{a_{\alpha}}}(G_{a_{\alpha-1}})\]
where we define $G_{a_{-1}}=\{1_G\}$ and consequently $\mathbf{C}_{G_{a_{0}}}(G_{a_{-1}})=G_{a_{0}}$. Let $\mathbf{g}=(g_{\alpha}: \alpha<\kappa \text{ and } \alpha \text{ is non-limit})$ be an element of $\mathcal{G}_A$. We will recursively construct a sequence $(\Phi_{\alpha}^{\mathbf{g}})_{\alpha <\kappa}$ of automorphisms of $G$ such that
$$(\dagger)\ \ \ \Phi^{\mathbf{g}}_{\alpha}[G_{a_{\delta}}]=G_{a_{\delta}} \text{ and }\Phi^{\mathbf{g}}_{\beta} \upharpoonright G_{a_{\beta}}= \Phi^{\mathbf{g}}_{\alpha} \upharpoonright G_{a_{\beta}} \text{ for all }\beta<\alpha \leq \delta<\kappa$$
The construction here is inspired by \cite[Theorem 6.8]{KW73}, however, its generalization will bring some subtle issues as we shall see. Our recursive construction is as follows.
\begin{itemize}
\item For $\alpha=0$, set $\Phi_0^{\mathbf{g}}=\iota_{g_0}$. Then $\Phi_0^{\mathbf{g}}$ satisfies $(\dagger)$.
\item Let $\alpha<\kappa$ and suppose that such maps $(\Phi_{\beta}^{\mathbf{g}})_{\beta \leq \alpha}$ satisfying the conditions have been constructed. Set $$\Phi_{\alpha+1}^{\mathbf{g}}=\iota_{g_{\alpha+1}} \circ \Phi_{\alpha}^{\mathbf{g}}$$ We have $\Phi_{\alpha+1}^{\mathbf{g}} \upharpoonright G_{a_{\alpha}} = \Phi^{\mathbf{g}}_{\alpha} \upharpoonright G_{a_{\alpha}}$ because $g_{\alpha+1} \in \mathbf{C}_{G_{a_{\alpha+1}}}(G_{a_{\alpha}})$. Together with the inductive assumption, this shows that $\Phi_{\alpha+1}^{\mathbf{g}}$ satisfies $(\dagger)$.
\item Let $\gamma<\kappa$ be limit and suppose that such maps $(\Phi_{\beta}^{\mathbf{g}})_{\beta < \gamma}$ satisfying the conditions have been constructed. Observe that the map
$$\Psi^{\mathbf{g}}_{\gamma}=\lim_{\alpha<\gamma} \left(\Phi_{\alpha}^{\mathbf{g}} \upharpoonright G_{a_{\alpha}}\right)$$
is an automorphism of $G_{a_{\gamma}}$ and so we can view $\Psi^{\mathbf{g}}_{\gamma} \in Sym(G_{a_{\gamma}})$ as an element $\psi^{\mathbf{g}}_{\gamma} \in G_{a_{\gamma}+1}$. Set $\Phi^{\mathbf{g}}_{\gamma}=\iota_{\psi^{\mathbf{g}}_{\gamma}}$. Then, as $G_{a_{\gamma}}$ is embedded into $Sym(G_{a_{\gamma}}) = G_{a_{\gamma}+1}$ by its right regular representation, we have that
\[\Phi^{\mathbf{g}}_{\gamma}(x)=\iota_{\psi^{\mathbf{g}}_{\gamma}}(x)=\Psi^{\mathbf{g}}_{\gamma}(x)=\Phi_{\alpha}^{\mathbf{g}}(x)\]
for all $x \in G_{a_{\alpha}}$ and
for all $\alpha < \gamma$. Thus $(\dagger)$ is satisfied for $\Phi^{\mathbf{g}}_{\gamma}$.
\end{itemize}
Having constructed such a sequence  $(\Phi_{\alpha}^{\mathbf{g}})_{\alpha <\kappa}$, it is easily checked that the map
\[ \Phi^{\mathbf{g}}=\lim_{\alpha<\kappa} \left(\Phi_{\alpha} \upharpoonright G_{a_{\alpha}}^{\mathbf{g}}\right)\]
is an $A$-level preserving automorphism of $G$.

We shall now endow the set
\[ \prod_{\substack{\alpha  <\kappa\\\alpha \text{ is non-limit}}} \mathbf{C}_{G_{a_{\alpha}}}(G_{a_{\alpha-1}})\]
with the binary operation given by
\begin{align*}
\mathbf{g} \cdot \mathbf{h} &=(g_{\alpha}: \alpha<\kappa \text{ and } \alpha \text{ is non-limit})\ \cdot\ (h_{\alpha}: \alpha<\kappa \text{ and } \alpha \text{ is non-limit})\\
&=(g_{\alpha}\Phi^{\mathbf{g}}_{\alpha-1}(h_{\alpha}): \alpha<\kappa \text{ and } \alpha \text{ is non-limit})
\end{align*}
where we define $\Phi^{\mathbf{g}}_{-1}(h_{0})=h_{0}$. As expected, we will obtain a group structure. On the other hand, in order to prove that this operation is associative, we need the following fact.

\begin{proposition}\label{mainlemma} $\Phi^{\mathbf{g} \cdot \mathbf{h}}_{\alpha}=\Phi^{\mathbf{g}}_{\alpha} \circ \Phi^{\mathbf{h}}_{\alpha}$ for all $\alpha<\kappa$.
\end{proposition}
\begin{proof}We shall prove this by transfinite induction. The claim holds for $\alpha=0$ since
\[ \Phi^{\mathbf{g} \cdot \mathbf{h}}_{0}=\iota_{g_0 \cdot h_0}=\iota_{g_0} \circ \iota_{h_0} = \Phi^{\mathbf{g}}_{0} \circ \Phi^{\mathbf{h}}_{0}\]
Let $\alpha < \kappa$ and suppose that the claim holds for $\alpha<\kappa$. Then, by the induction assumption, we have that
\begin{align*}
\Phi^{\mathbf{g} \cdot \mathbf{h}}_{\alpha+1}&= \iota_{\left(g_{\alpha+1} \cdot \Phi^{\mathbf{g}}_{\alpha}(h_{\alpha+1})\right)} \circ \Phi^{\mathbf{g} \cdot \mathbf{h}}_{\alpha}\\
&=\iota_{g_{\alpha+1}} \circ \iota_{\Phi^{\mathbf{g}}_{\alpha}(h_{\alpha+1})} \circ \Phi^{\mathbf{g}}_{\alpha} \circ \Phi^{\mathbf{h}}_{\alpha}\\
&=\iota_{g_{\alpha+1}} \circ \Phi^{\mathbf{g}}_{\alpha} \circ \iota_{h_{\alpha+1}} \circ\Phi^{\mathbf{h}}_{\alpha}\\
&=\Phi^{\mathbf{g}}_{\alpha+1} \circ \Phi^{\mathbf{h}}_{\alpha+1}
\end{align*}
Let $\gamma < \kappa$ be limit and suppose that the claim holds for all $\alpha < \gamma$. Then, by the induction assumption, we have $\Psi^{\mathbf{g} \mathbf{h}}_{\gamma}=\Psi^{\mathbf{g} }_{\gamma} \circ \Psi^{\mathbf{h}}_{\gamma}$ where $\Psi$ are as in the inductive construction. It follows that \[\Phi^{\mathbf{g} \cdot \mathbf{h}}_{\gamma}=\iota_{\psi^{\mathbf{g}\mathbf{h}}_{\gamma}}=\iota_{\psi^{\mathbf{g}}_{\gamma} \psi^{\mathbf{h}}_{\gamma}}=
\iota_{\psi^{\mathbf{g}}_{\gamma}} \circ \iota_{\psi^{\mathbf{h}}_{\gamma}}=\Phi^{\mathbf{g}}_{\gamma} \circ \Phi^{\mathbf{h}}_{\gamma}
\]
Thus the claim holds for $\gamma$ which completes the inductive proof.\end{proof}

Using Proposition \ref{mainlemma}, it is tedious but straightforward to check that the set
\[ \mathcal{G}_A=\prod_{\substack{\alpha  <\kappa\\\alpha \text{ is non-limit}}} \mathbf{C}_{G_{a_{\alpha}}}(G_{a_{\alpha-1}})\] together with the operation defined above indeed forms a group. The group $\mathcal{G}_A$ is supposed to be the transfinite external semi-direct product
\[ \displaystyle \left(\left(\dots\left(\left(G_{a_0}\ _{\Phi_{0}}\ltimes \mathbf{C}_{G_{a_1}}(G_{a_0})\right)\ _{\Phi_{1}}\ltimes \mathbf{C}_{G_{a_2}}(G_{a_1})\right)\dots\right)\ _{\Phi_{\omega}}\ltimes \mathbf{C}_{G_{a_{\omega+1}}}(G_{a_{\omega}})\right)...\]
which can formally be constructed as the inverse limit of semi-direct products defined appropriately. It turns out that most of these semi-direct products are indeed direct products as implied by the following proposition.

\begin{proposition}\label{semidirectstrivialize} Let $\alpha <\kappa$ be limit. Then
$$\Phi^{\mathbf{g}}_{\alpha+n}(h_{\alpha+m})=h_{\alpha+m}$$
for all $1 \leq n < m <\omega$.
\end{proposition}
\begin{proof} We will prove this by induction on $n$. Let $1 < m <\omega$. Since we have $h_{\alpha+m} \in \mathbf{C}_{G_{a_{\alpha+m}}}(G_{a_{\alpha+m-1}}) \subseteq \mathbf{C}_{G_{a_{\alpha+m}}}(G_{a_{\alpha+1}})$, we obtain
\[\Phi^{\mathbf{g}}_{\alpha+1}(h_{\alpha+m})=\left(\iota_{g_{\alpha+1}} \circ \Phi^{\mathbf{g}}_{\alpha}\right)(h_{\alpha+m})=\left(\iota_{g_{\alpha+1}} \circ \iota_{\psi^{\mathbf{g}}_{\alpha}}\right)(h_{\alpha+m})=h_{\alpha+m}\]
and hence the claim holds for $n=1$. Let $1 \leq n < \omega$ and suppose that the claim holds for $n$. Let $n+1 < m < \omega$. Then we have
\[\Phi^{\mathbf{g}}_{\alpha+n+1}(h_{\alpha+m})=\left(\iota_{g_{\alpha+n+1}} \circ \Phi^{\mathbf{g}}_{\alpha+n}\right)(h_{\alpha+m})=\iota_{g_{\alpha+n+1}}(h_{\alpha+m})=h_{\alpha+m}\]
because $h_{\alpha+m} \in \mathbf{C}_{G_{a_{\alpha+m}}}(G_{a_{\alpha+m-1}}) \subseteq \mathbf{C}_{G_{a_{\alpha+m}}}(G_{a_{\alpha+n+1}})$. Thus the claim holds for $n+1$ which completes the inductive proof.\end{proof}

It is also easily shown that $\Phi^{\mathbf{g}}_{0}(h_{m})=h_{m}$ for all $1 \leq m < \omega$. Therefore the semi-direct products at indices that are not successors of limits are indeed direct products. The reason that this does not generalize to all indices is that, for a limit ordinal, for example $\omega$, we have no reason to have $$\Phi^{\mathbf{g}}_{\omega}(h_{\omega+1})=\iota_{\psi^{\mathbf{g}}_{\omega}}(h_{\omega+1})=h_{\omega+1}$$ because $\psi^{\mathbf{g}}_{\omega} \in G_{a_{\omega}+1}$ need not commute with $h_{\omega+1} \in G_{a_{\omega}+1}$.

We are now ready to prove the first main theorem of this section. Consider the map $\Theta: \mathcal{G}_A \rightarrow Aut_A(G)$ given by $\Theta(\mathbf{g})=\Phi^{\mathbf{g}}$.

\begin{theorem}\label{mainembeddingtheorem} $\Theta$ is a monomorphism from $\mathcal{G}_A$ to $Aut_A(G)$.
\end{theorem}
\begin{proof} Let $\mathbf{g}, \mathbf{h} \in \mathcal{G}_A$. Since we have
\[ \Phi^{\mathbf{g} \cdot \mathbf{h}}_{\alpha} =\left(\Phi^{\mathbf{g}}_{\alpha} \circ \Phi^{\mathbf{h}}_{\alpha}\right)\]
and $\Phi^{\cdot}$ is defined as the limit of the restrictions $\Phi^{\cdot}_{\alpha} \upharpoonright G_{a_{\alpha}}$ each of which extends the previous ones, it easily follows from Proposition \ref{mainlemma} that $\Theta$ is a homomorphism.

It remains to prove that $\Theta$ is injective. Recall that each $G_{a_{\alpha}}$ is centerless. Now suppose that $\mathbf{g} \neq \mathbf{h}$. Take the least $\alpha$ such that $g_{\alpha} \neq h_{\alpha}$. If $\alpha=0$, then we have $\Theta(\mathbf{g})\neq\Theta(\mathbf{h})$ because 
\[ \Phi^{\mathbf{g}}(x)=\Phi^{\mathbf{g}}_0(x)=\iota_{g_0}(x) \neq \iota_{h_0}(x)=\Phi^{\mathbf{h}}_0(x)=\Phi^{\mathbf{h}}(x)\]
for some $x \in G_{a_0}$. This follows from that $G_{a_0}$ being centerless implies that distinct elements induce distinct inner automorphisms. Suppose $\alpha \neq 0$. Then $\alpha=\beta+1$ and it follows from the minimality of $\alpha$ that $\Phi^{\mathbf{g}}_{\beta}=\Phi^{\mathbf{h}}_{\beta}$. As before, $G_{a_{\alpha}}$ being centerless implies that there exists $x \in G_{a_\alpha}$ such that $\iota_{g_\alpha}(x) \neq \iota_{h_\alpha}(x)$. Then we have $y=\left(\Phi^{\mathbf{g}}_{\beta}\right)^{-1}(x) \in G_{a_{\alpha}}$ and so
\begin{align*}
\Phi^{\mathbf{g}}\left(y\right)=\Phi_{\alpha}^{\mathbf{g}}\left(y\right)&=(\iota_{g_{\alpha}} \circ \Phi^{\mathbf{g}}_{\beta})\left(\left(\Phi^{\mathbf{g}}_{\beta}\right)^{-1}(x)\right)\\
&=\iota_{g_{\alpha}}(x)\\
&\neq \iota_{h_{\alpha}}(x)\\
&=(\iota_{h_{\alpha}} \circ \Phi^{\mathbf{h}}_{\beta})\left(\left(\Phi^{\mathbf{h}}_{\beta}\right)^{-1}(x)\right)=\Phi_{\alpha}^{\mathbf{h}}\left(y\right)=\Phi^{\mathbf{h}}\left(y\right)
\end{align*}
Thus $\Theta(\mathbf{g})\neq\Theta(\mathbf{h})$ and so $\Theta$ is a group embedding.
\end{proof}

The second main theorem of this section is that $\Theta$ is indeed an isomorphism under additional mild assumptions on $A$.

\begin{theorem}\label{mainsurjectiontheorem} $\Theta$ is an isomorphism from $\mathcal{G}_A$ onto $Aut_A(G)$ provided that $a_{\alpha}$ is not limit whenever $\alpha$ is not limit.
\end{theorem}
\begin{proof} By Theorem \ref{mainembeddingtheorem}, it suffices to prove that $\Theta$ is surjective. Let $\varphi \in Aut_A(G)$. By assumption, $G_{a_\alpha}=Sym(G_{a_{\alpha}-1})$ is a symmetric group on more than $6$ elements for non-limit $\alpha$, and hence only has inner automorphisms. Therefore we can find a sequence $$\mathbf{h}=(h_{\alpha}: \alpha<\kappa \text{ and } \alpha \text{ is non-limit})$$ such that $h_{\alpha} \in G_{a_\alpha}$ and
\[ \iota_{h_\alpha} \upharpoonright G_{a_\alpha}= \varphi \upharpoonright G_{a_\alpha}\]
for all non-limit $\alpha<\kappa$. We now produce another sequence $$ \mathbf{g}=(g_{\alpha}: \alpha<\kappa \text{ and } \alpha \text{ is non-limit})$$ as follows.
\begin{itemize}
\item For $\alpha=0$, set $g_0=h_0$.
\item Let $\alpha<\kappa$ be non-limit. Set $g_{\alpha+1}=h_{\alpha+1} h_{\alpha}^{-1}$.
\item Let $\gamma<\kappa$ be a limit. Then $\varphi \upharpoonright G_{a_\gamma} \in Sym(G_{a_\gamma})$ and hence it can be viewed as an element $t_{\gamma} \in G_{a_\gamma+1}$. Set $g_{\gamma+1}=h_{\gamma+1}t_{\gamma}^{-1}$.
\end{itemize}
It is straightforward to check that $g_{\alpha} \in \mathbf{C}_{G_{a_\alpha}}(G_{a_{\alpha-1}})$ and so $\mathbf{g} \in \mathcal{G}_A$. We will next show that $\Theta(\mathbf{g})=\varphi$ which completes the proof that $\Theta$ is an isomorphism.

We claim that $\Phi^{\mathbf{g}}_{\alpha} \upharpoonright G_{a_{\alpha}} = \varphi \upharpoonright G_{a_\alpha}$ for all $\alpha<\kappa$. Assume not and take the least $\delta<\kappa$ for which this equality fails. Observe that, by construction of $\Phi^{\mathbf{g}}$, this equality holds for a limit ordinal $\gamma$ whenever it holds for ordinals less than $\gamma$. Moreover, it holds for $\alpha=0$. Therefore $\delta$ has to be a successor ordinal, say, $\delta=\gamma+n$ where $\gamma$ is limit or $0$, and $1 \leq n < \omega$. If $\gamma$ is limit, then it follows from the construction that
\begin{align*}
\Phi^{\mathbf{g}}_{\gamma+n}&=\iota_{g_{\gamma+n}} \circ \iota_{g_{\gamma+n-1}} \circ \dots \iota_{g_{\gamma+1}} \circ \Phi^{\mathbf{g}}_{\gamma}\\
&=\iota_{g_{\gamma+n}g_{\gamma+n-1}\dots g_{\gamma+1}} \circ \Phi^{\mathbf{g}}_{\gamma}=\iota_{h_{\gamma+n} t_{\gamma}^{-1}} \circ \iota_{t_{\gamma}}=\iota_{h_{\gamma+n}}
\end{align*}
If $\gamma=0$, then a similar argument gives that $\Phi^{\mathbf{g}}_{\gamma+n}=\iota_{h_{\gamma+n}}$. Thus $\Phi^{\mathbf{g}}_{\gamma+n} \upharpoonright G_{a_{\gamma+n}} = \iota_{h_{\gamma+n}} \upharpoonright G_{a_{\gamma+n}} = \varphi \upharpoonright G_{a_{\gamma+n}}$ which is a contradiction. Therefore there exists no such $\delta$ and hence $\Theta$ is an isomorphism.
\end{proof}

We would like to say a few words about the additional hypothesis in Theorem \ref{mainsurjectiontheorem}. Suppose that $A$ is the set of limit ordinals below $\kappa$. Then $Aut_A(G)$ already contains the groups $Aut(G_{\omega}),Aut(G_{\omega\cdot 2})$ etc. which have pretty complicated structures themselves. Thus a structure theorem seems out of reach without further assumptions on $A$ which simplify the structure of the automorphism group of each level.

\section{Cardinality of automorphism groups}

In this section, we shall next see some important corollaries of our results in the previous section.

\begin{corollary}\label{maincorollary} Let $\kappa$ be inaccesible and let $K$ be the unique $\kappa$-existentially closed group of cardinality $\kappa$. Then we have a monomorphism $\mathcal{G}_{\kappa} \hookrightarrow Aut(K)$. In particular, we have $|Aut(K)|=2^{\kappa}$.
\end{corollary}
\begin{proof} By \cite[Theorem 2.7, Section 4]{KK18}, the group $G$ constructed in this section when $\kappa$ is taken to be inaccessible is isomorphic the unique $\kappa$-existentially closed group $K$ of cardinality $\kappa$. By Theorem \ref{mainsurjectiontheorem}, we have $\mathcal{G}_{\kappa} \cong Aut_{\kappa}(G) \hookrightarrow Aut(K)$. As the centralizer in each component of $\mathcal{G}_{\kappa}$ has at least two elements, we have that $2^{\kappa} \leq |\mathcal{G}_{\kappa}| \leq |Aut(K)| \leq {\kappa}^{\kappa}=2^{\kappa}$.
\end{proof}

Observe that, at its core, our argument has something to do with the group being a limit of regular representations but nothing to do with it being $\kappa$-existentially closed. For this reason, we cannot apply it to $\kappa$-existentially closed groups that are not limits of regular representations. This brings up the following question.

\textbf{Question.} Let $G$ be a $\kappa$-existentially closed group of cardinality $\lambda \geq \kappa$. Is it true that $|Aut(G)|=2^{\lambda}$?

An analysis of the proof of Theorem \ref{mainembeddingtheorem} shows that it does not rely on the assumption that $|G|=\kappa$ stated at the very beginning of this subsection. Consequently, Theorem \ref{mainembeddingtheorem} holds for any limit of regular representations of length $\kappa$ as long as $G_{a_0}$ is centerless. Similarly, Theorem \ref{mainsurjectiontheorem} holds for any limit of regular representations of length $\kappa$ as long as $G_{a_0}$ is a complete group. Before we state the variation of Corollary \ref{maincorollary} in this context, let us recall the definition of beth numbers defined recursively as follows.
\begin{itemize}
\item $\beth_0=\aleph_0$,
\item $\beth_{\alpha+1}=2^{\beth_{\alpha}}$ for all $\alpha$ and
\item $\beth_{\gamma}=\sup\{\beth_{\alpha}:\alpha<\gamma\}$ for all limit $\gamma$.
\end{itemize}
We will need the following elementary identity in cardinal arithmetic for which we were unable to find a suitable reference. So we include its proof for completeness.
\begin{lemma}\label{cardinalarithmeticlemma} Let $\kappa$ be an infinite cardinal. Then we have
\[ \beth_{\kappa+1} = \prod_{\alpha<\kappa} \beth_{\alpha}\]
\end{lemma}
\begin{proof} Fix some bijection $g_{\alpha}: \mathcal{P}(\beth_{\alpha-1}) \rightarrow \beth_{\alpha}$ for each successor ordinal $\alpha<\kappa$. Then the map $X \mapsto f_X$ from $\mathcal{P}(\beth_{\kappa})$ to $\prod_{\alpha<\kappa} \beth_{\alpha}$ where
 \[ f_X(\alpha)=\begin{cases} g_{\alpha}(X \cap \beth_{\alpha-1}) & \text{ if } \alpha \text{ is a successor}\\
0 & \text{otherwise}\end{cases}\]
is an injection. It follows that
\[\beth_{\kappa+1} = |\mathcal{P}(\beth_{\kappa})|\leq  \prod_{\alpha<\kappa} \beth_{\alpha} \leq \prod_{\alpha<\kappa} \beth_{\kappa} = \beth_{\kappa}^{\kappa} \leq \left(2^{\beth_\kappa}\right)^\kappa=2^{\beth_{\kappa}}=\beth_{\kappa+1}\]\end{proof}

We are now ready to state the other main corollary of our results.

\begin{corollary}\label{maincorollary2} Let $\kappa$ be an uncountable cardinal and $G$ be a limit of regular representations of length $\kappa$ with countable base. Then we have $$|Aut(G)|=\beth_{\kappa+1}$$
\end{corollary}
\begin{proof} Suppose that $G$ is the direct limit of $(G_{\alpha})_{\alpha<\kappa}$. It is not difficult to prove by induction that $\beth_{\alpha} \leq |G_{\omega+\alpha}| \leq \beth_{\omega+\alpha}$ for all $\alpha<\kappa$. (We actually have equality on the left if the base is finite, and equality on the right if the base is countably infinite.) Hence $|G|=\beth_{\kappa}$. We also have the embedding $G_{\alpha-1} \hookrightarrow \mathbf{C}_{G_{\alpha}}(G_{\alpha-1})$ for all non-limit $\alpha<\kappa$ since the image of left regular representation of $G_{\alpha-1}$ commutes with the image of its right regular representation. Therefore, applying Theorem \ref{mainembeddingtheorem} and Lemma \ref{cardinalarithmeticlemma} with $A=\kappa-\{0\}$, we obtain that
\begin{align*}
\beth_{\kappa+1} = \prod_{\alpha<\kappa} \beth_{\alpha} & \leq \bigg| \prod_{\alpha<\kappa} G_{\omega+\alpha} \bigg|\\
& \leq \bigg| \prod_{\alpha<\kappa} \mathbf{C}_{G_{\omega+\alpha+1}}(G_{\omega+\alpha}) \bigg|\\
& \leq \bigg| \prod_{\beta<\kappa} \mathbf{C}_{G_{\beta+1}}(G_{\beta}) \bigg|\\
& \leq |\mathcal{G}_{A}| = |Aut_{A}(G)| \leq |Aut(G)| \leq 2^{|G|}=\beth_{\kappa+1}
\end{align*}
\end{proof}

We would like to note that the proofs of Theorem \ref{mainembeddingtheorem} and \ref{mainsurjectiontheorem} together with Proposition \ref{mainlemma} can also be applied to investigate the automorphisms of Hall's universal group $G_{\omega}$. In particular, letting $\kappa=\omega$ and applying these arguments together with the fact $\mathbf{C}_{G_n}(G_{n-1}) \cong G_{n-1}$, we obtain the following theorem which was already implicitly contained in \cite[Theorem 6.8]{KW73}
\begin{corollary} $Aut_{\omega}(G_{\omega})$ is isomorphic to
\[ Sym(7) \times \displaystyle \prod_{i=0}^{\infty} Sym(n_i)\]
where $n_0=7$ and $n_{i+1}=n_i!$ for all $i \in \mathbb{N}$.
\end{corollary}


\begin{thebibliography}{KKK18}

\bibitem[HS88]{HigmanScott88}
Graham Higman and Elizabeth Scott, \emph{Existentially closed groups}, London
  Mathematical Society Monographs. New Series, vol.~3, The Clarendon Press,
  Oxford University Press, New York, 1988, Oxford Science Publications.

\bibitem[Jec03]{Jech03}
Thomas Jech, \emph{Set theory}, Springer Monographs in Mathematics,
  Springer-Verlag, Berlin, 2003, The third millennium edition, revised and
  expanded.
  
\bibitem[KKK18]{KKK18}
Burak Kaya, Otto~H. Kegel, and Mahmut Kuzucuo\u{g}lu, \emph{On the existence of
  {$\kappa$}-existentially closed groups}, Arch. Math. (Basel) \textbf{111}
  (2018), no.~3, 225--229.

\bibitem[Keg09]{Kegel09}
Otto~H. Kegel, \emph{Regular limits of infinite symmetric groups}, Ischia group
  theory 2008, World Sci. Publ., Hackensack, NJ, 2009, pp.~120--130.

\bibitem[KK18]{KK18}
Otto~H. Kegel and Mahmut Kuzucuo\u{g}lu, \emph{{$\kappa$}-existentially closed
  groups}, J. Algebra \textbf{499} (2018), 298--310.


\bibitem[KW73]{KW73}
Otto~H. Kegel and Bertram A.~F. Wehrfritz, \emph{Locally finite groups},
  North-Holland Publishing Co., Amsterdam-London; American Elsevier Publishing
  Co., Inc., New York, 1973, North-Holland Mathematical Library, Vol. 3.

\bibitem[Neu37]{Neumann37}
B.~H. Neumann, \emph{Some remarks on infinite groups}, J. London Math. Soc.
  (1937), 120--127.

\bibitem[Neu73]{Neumann73}
\bysame, \emph{The isomorphism problem for algebraically closed groups}, Word
  problems: decision problems and the {B}urnside problem in group theory
  ({C}onf. on {D}ecision {P}roblems in {G}roup {T}heory, {U}niv. {C}alifornia,
  {I}rvine, {C}alif. 1969; dedicated to {H}anna {N}eumann), 1973, pp.~553--562.
  Studies in Logic and the Foundations of Math.,Vol. 71.

\bibitem[Rab77]{Rabin77}
E.~B. Rabinovi\v{c}, \emph{Limits of finite symmetric groups. {II}}, Vesc\={\i}
  Akad. Navuk BSSR Ser. F\={\i}z.-Mat. Navuk (1977), no.~5, 99--102, 142.

\bibitem[Sco51]{Scott51}
W.~R. Scott, \emph{Algebraically closed groups}, Proc. Amer. Math. Soc.
  \textbf{2} (1951), 118--121.

\end{thebibliography}
\end{document}